\documentclass[reqno]{amsart}

\usepackage{amsmath}
\usepackage{amsfonts}
\usepackage{amssymb}
\usepackage{amsthm}
\usepackage{mathenv}
\usepackage{enumerate}
\usepackage{graphicx}

\theoremstyle{plain} 
\newtheorem{theorem}{Theorem}[section]
\newtheorem{prop}[theorem]{Proposition}
\newtheorem{lem}[theorem]{Lemma}

\theoremstyle{definition} 
\newtheorem{defn}[theorem]{Definition}
\newtheorem{rem}[theorem]{Remark}

\newtheorem{conj}[theorem]{Conjecture}

\newcommand{\norm}[2]{\left\| {#1}\right\| _{#2}}
\newcommand{\abs}[2][]{\ensuremath{\left \lvert #2 \right \rvert_{#1}}} 

\newcommand{\C}[1]{\mathcal{#1}}
\newcommand{\B}[1]{\mathbb{#1}}

\newcommand{\oemph}[1]{\textit{#1}}

\newcommand{\rst}[1]{\ensuremath{{\mathbin\upharpoonright}%
\raise-.5ex\hbox{$#1$}}} 

\newcommand{\all}[2]{\{\,{#1}\,:\,{#2}\,\}} 
\newcommand{\cl}[1]{{\rm cl} ( #1 )} 
\newcommand{\cc}{\subset\!\!\!\subset}
\newcommand{\tr}{{\mathcal L}}
\newcommand{\e}[1]{{e}^{#1}}

\newcommand{\image}[1]{\operatorname{im}(#1)}

\title{Analytic expanding circle maps with explicit spectra}
\date{31 May 2013}
\author{Julia Slipantschuk, Oscar F.~ Bandtlow, and Wolfram Just}
\address{School of Mathematical Sciences, Queen Mary University of London, Mile End Road,
London E1 4NS, UK}
\email{j.slipantschuk@qmul.ac.uk,\, o.bandtlow@qmul.ac.uk, \,w.just@qmul.ac.uk}
\begin{document}


\begin{abstract}
We show that for any $\lambda \in \mathbb{C}$ with $|\lambda|<1$ there exists an analytic 
expanding circle map such that the eigenvalues of the 
associated 
transfer 
operator (acting on holomorphic functions) are precisely the nonnegative powers of 
$\lambda$ and $\overline{\lambda}$.
As a consequence we obtain a counterexample to a variant of a
conjecture of Mayer on the reality of spectra of transfer
operators. 
\end{abstract}

\maketitle

\section{Introduction}\label{sec:Intro} 
Transfer operators can be regarded as global representations
of a system's dynamics. Their spectral properties yield insight into
the 
statistical long-term behaviour of the underlying system, 
such as rates of mixing or the existence of central limit 
theorems (see, for example, \cite{Bala:00,BoyaGora:97}).

If $\{\phi_k\colon k=1,\ldots, K \}$ is the set of local inverse branches of a real
analytic expanding map, 
then the associated transfer operator $\mathcal{L}$, defined by 
\begin{equation}\label{eq:L_intro} 
 (\mathcal{L}f)(z) = \sum_{k=1}^{K} \phi'_k(z)(f\circ\phi_k)(z),
\end{equation} 
preserves and acts compactly on certain spaces of holomorphic functions 
(see, for example, \cite{BaJe_AM08, BaJe_LMS, Mayer1991, ruelle76}).\footnote{Often 
$\mathcal{L}$ is referred to as the Perron-Frobenius operator. 
More general transfer operators can be obtained by 
replacing the $\phi_k'$ in \eqref{eq:L_intro} by   
other suitable weight functions.} 
In particular, its spectrum is a sequence of 
eigenvalues $\{\lambda_n(\mathcal{L})\}$ 
converging to zero, together with
zero itself.\footnote{Under mild assumptions it
  is possible to show that the spectrum does not depend on the
  particular choice of holomorphic finction space, see
  \cite{BaJe_ETDS}.} 

Assuming that the eigenvalues of $\mathcal L$ are 
ordered with respect to decreasing modulus, the second eigenvalue $\lambda_2(\mathcal{L})$ determines the exponential rate of mixing for 
generic analytic observables. 
A faster exponential
rate of mixing can occur if one chooses (non-generic) observables which do not `feel' the rate corresponding to the first
$n$ subleading eigenvalues, that is, observables in a subspace complementary to the 
eigendirections of $\{\lambda_2(\mathcal{L}),\ldots,\lambda_{n+1}(\mathcal{L})\}$.
Consequently, knowledge of the whole spectrum $\sigma(\mathcal{L})$ is useful, since it  
determines the set of all possible exponential  mixing rates, also known as the correlation spectrum in this context. 


Surprisingly there exist only few examples of maps in the literature 
for which 
the spectrum of the corresponding transfer operator is known
explicitly. 
The only one-dimensional examples known to the authors are piecewise linear 
interval maps, or more generally, piecewise linear 
Markov interval maps 
(see, for example, \cite{SBJ}) and circle maps of the 
form $z\mapsto z^n$ for $n$ a non-zero 
integer. It turns out that for the latter maps 
the spectra of the corresponding transfer operators 
when acting on analytic functions are all identical and coincide with 
the two-point set $\{0,1\}$. See \cite[Exercise 2.15]{Bala:00} for a
proof when $n=2$; the general case can be proved along the same
lines. 





As a result, the question arises as to whether 
there are analytic circle maps for which the associated transfer operator 
has infinitely many non-zero eigenvalues.


The purpose of this paper is twofold: firstly, 
to construct examples of circle maps with 
infinitely many
explicit non-zero eigenvalues in the spectrum of $\mathcal{L}$.  
Denoting the circle by $\mathbb{T}=\{z\in \mathbb{C}:|z|=1\}$, 
and letting ${H^\infty}(A)$ 
be the space of bounded holomorphic functions on an annulus $A$ 
containing $\mathbb{T}$ (see Section 
\ref{sec:Rigorous} for precise definitions), we can state the following.
\begin{theorem}\label{thm:main}
For any $\lambda \in \mathbb{C}$ with $|\lambda|<1$ there exists an analytic expanding circle map 
$\tau$ such that the eigenvalues of the associated transfer operator $\mathcal{L}$ in 
\eqref{eq:L_intro}, when acting on ${H^\infty}(A)$, are precisely all nonnegative powers of 
$\lambda$ and $\overline{\lambda}$, that is, the spectrum of $\mathcal{L}$ is
\[ \sigma(\tr) = \all{\lambda^n}{n\in\mathbb{N}_0}\cup
\all{\overline{\lambda}^n}{n\in\mathbb{N}}\cup \{0\}\,.
\]
Moreover, the algebraic multiplicity of the leading eigenvalue is one,
while the algebraic multiplicity of the remaining eigenvalues is two
for real $\lambda$ and one for $\lambda$ with nonvanishing imaginary part.
\end{theorem}   

Secondly, the analytic maps arising from Theorem~\ref{thm:main} 
yield interval 
maps\footnote{Throughout the paper we use lower case Greek letters to
  denote inverse branches of circle maps and the corresponding upper
  case letters to denote inverse branches of interval maps.} 
which
provide counterexamples to the following conjecture.
\begin{conj}[Weak variant of Mayer's conjecture in dimension one]
\label{conj: LevinMayer}
Let $\Omega \subset \mathbb{C}$ be a bounded domain with 
$\Omega_\mathbb{R} = \Omega\cap \mathbb{R} \neq \emptyset$ and 
$\Phi_k:\Omega \to \Omega$ 
contracting holomorphic mappings with their unique fixed 
points $z_k^*$ in $\Omega_\mathbb{R}$.
If the $\Phi'_k(z_k^*)$ are real, then all eigenvalues of 
the corresponding transfer operator 
\begin{equation*} 
 (\mathcal{L}f)(z) = \sum_{k=1}^{K} \Phi'_k(z)(f\circ\Phi_k)(z)
\end{equation*} 
with small enough modulus are real.
\end{conj} 
Mayer~\cite{Mayer1987} originally conjectured that transfer
operators satisfying the hypotheses of the above conjecture have real
spectra. Counterexamples to Mayer's conjecture were given by 
Levin in~\cite{Levin1994} which led to the above weakening of Mayer's
conjecture. 

Before proceeding to definitions and rigorous proofs, 
we shall briefly explain the genesis of our result.
Turning the classical question of finding eigenvalues of $\tr$ for 
a given map on its head, we attempt to construct
a map for which the transfer operator has a given 
eigenvalue and a given eigenfunction.
Considering an analytic expanding map on the interval $I$ with two full
branches, the eigenvalue problem of the transfer operator formally reads
\begin{equation}\label{eq:EVeq}
  \mu_n u_n = \Phi_1'\cdot (u_n \circ \Phi_1) + \Phi_2' \cdot (u_n \circ \Phi_2),       
\end{equation}
where $\mu_n$ and $u_n$ are an eigenvalue and eigenfunction of $\tr$.

Given an analytic invariant density $\rho$ one may consider \eqref{eq:EVeq}
for $n=0$ with $\mu_0=1$ and $u_0=\rho$ as an equation
to compute suitable inverse branches $\Phi_1$ and $\Phi_2$ of the map.
This setup is a particularly simple
case of the so called inverse Perron-Frobenius problem 
\cite{Ershov1988,Gora1993}, which has
been applied in various guises to taylor-make 
chaotic maps with given
stationary
properties (see, for example, \cite{Diakonos1996}). 
As we are attempting to construct a map with two branches we are at
liberty to specify 
a nontrivial eigenvalue and corresponding eigenfunction. 
Thus, given an invariant
density $\rho$, a
real number $\lambda$ with $|\lambda|<1$, and
a potential eigenfunction $u$, we seek to solve
\begin{equation}\label{eq:PU}
 \begin{split}
 P =  P\circ \Phi_1 + P \circ \Phi_2,  \\ 
  \lambda U = U\circ\Phi_1 + U\circ\Phi_2  
 \end{split}
\end{equation}
for the inverse branches $\Phi_1$ and $\Phi_2$.
Here P and U denote suitable antiderivatives of $\rho$ and $u$, respectively.
A priori, there is no guarantee that \eqref{eq:PU} admits a
real solution for $\Phi_1$ and $\Phi_2$ and that such a solution
actually determines
an analytic full branch interval map. Developing conditions under
which this is the case seems to be a challenging task.
Nevertheless, if we fix the interval $I=[-1,1]$, take $\rho$ to be the 
uniform density,
and $u(x)=\cos(\pi x)$ for the
eigenfunction with eigenvalue $\lambda$,
then \eqref{eq:PU} leads to
\begin{equation}\label{eq:cos}
\begin{split}
 x &= \Phi_1(x)+\Phi_2(x),  \\
\lambda \sin(\pi x) &= \sin(\pi \Phi_1(x))+\sin(\pi \Phi_2(x)).
\end{split}
\end{equation}
After a short calculation it is possible to deduce 
from \eqref{eq:cos} an explicit expression for 
$\Phi_1$ and $\Phi_2$, and, as we shall see in Section~\ref{sec:circle}, this
indeed determines an analytic full branch map.
It turns out that for this particular example 
the complete spectrum can be obtained, an observation that provides 
the main content of Theorem~\ref{thm:main}.

Certainly, this reasoning does not generally provide examples with fully
prescribed spectrum. However, the method deserves further exploration, as it
furnishes maps with given partial spectral information.

The paper is structured as follows. 
In Section~\ref{sec:Rigorous} we introduce analytic expanding circle maps, and 
define Banach spaces of holomorphic functions on which the associated transfer 
operators $\mathcal{L}$ are compact.
Section~\ref{sec:circle} is devoted to the construction of a family of circle maps and 
the proof
of Theorem \ref{thm:main}. In Section~\ref{sec:Interval} we consider the same family of
maps on an interval and thus obtain counterexamples to Conjecture \ref{conj: LevinMayer}.

\section{Transfer operators for analytic circle maps}\label{sec:Rigorous}

The main purpose of this section is to define suitable function spaces on
which transfer operators induced by analytic expanding circle
maps are compact. 
We start by defining what is meant by an analytic expanding circle map.

\begin{defn}
 We  say that $\tau:\mathbb{T}\to\mathbb{T}$ is an \oemph{analytic expanding circle map} if the following two conditions hold:
\begin{enumerate}[(i)]
 \item $\tau$ is analytic on $\mathbb{T}$;
 \item $ \inf_{z\in \mathbb T}|\tau'(z)|>1$.
\end{enumerate}
\end{defn}
It is not difficult to see that $\tau$ is a $K$-fold covering of $\mathbb T$ for some integer $K>1$. Moreover, the map $\tau$ has analytic extensions to certain annuli containing $\mathbb T$. With slight abuse of notation we shall write $\tau$ for the various 
extensions as well.  
To be precise, for $r<1<R$ let $A_{r,R}$ denote the open annulus 
$A_{r,R} = \all{z \in \mathbb{C}}{r<|z|<R}$ and write  
\[ \mathcal A=\all{ A_{r,R}}{\text{$\tau$ is analytic on $A_{r,R}$}}\,. \]
The expansivity of $\tau$ yields the following result. 

\begin{lem}\label{lem:expand}
 If $\tau$ is an analytic expanding circle map, then there is $A_0\in \mathcal A$  such that 
 \begin{enumerate}[(a)]
 \item both $\tau$ and $1/\tau$ are analytic on the closure $\cl{A_0}$ of $A_0$; 
\item
   $\tau(\partial A_0) \cap \cl{A_0} = \emptyset$, 
 where $\partial A_0$ denotes the boundary of $A_0$. 
 \end{enumerate}
 \end{lem}
\begin{proof}
Since $\tau$ is an analytic expanding circle map it is possible to choose $A_1\in \mathcal A$ such that both $\tau$ and $1/\tau$ are analytic on 
$\cl{A_1}$ with 
\[ \alpha:=\inf_{z\in A_1}\abs{\tau'(z)}>1\,.\]
It is not difficult to see that $(\rho ,\theta)\mapsto \log \abs{\tau(\rho \e{i\theta})}$ is differentiable for all $(\rho,\theta)$ with 
$\rho \exp(i\theta)\in A_1$ and 
\begin{align}
\frac{\partial}{\partial \rho} \log \abs{\tau(\rho \e{i\theta})} &= 
   \Re \left ( \e{i\theta} \frac{\tau'(\rho  \e{i\theta})}{  \tau(\rho  \e{i\theta})} \right )\,,\label{e:rde} \\
   \frac{\partial}{\partial \theta } \log \abs{\tau(\rho \e{i\theta})} &= 
   -\Im \left ( \rho\e{i\theta} \frac{\tau'(\rho \e{i\theta})}{  \tau(\rho  \e{i\theta})} \right )\,,\label{e:phide} 
\end{align}     
where $\Re(z)$ and $\Im(z)$ denote the real and imaginary part of $z\in \mathbb C$.  Since $\tau$ leaves $\mathbb T$ 
invariant, equation (\ref{e:phide}) implies either  
\begin{equation}
\label{e:pos}
 \e{i\theta} \frac{\tau'( \e{i\theta})}{  \tau( \e{i\theta})} \geq \alpha 
\quad (\forall \theta \in \mathbb R)\,,
\end{equation}
or 
\begin{equation}
\label{e:neg}
\e{i\theta} \frac{\tau'( \e{i\theta})}{  \tau( \e{i\theta})} \leq -\alpha 
\quad (\forall \theta \in \mathbb R)\,. 
\end{equation}
Suppose now that (\ref{e:pos}) holds (the other case can be dealt with similarly). Fixing $\beta$ with 
$1<\beta <\alpha$ we can choose $A_{r,R}\in \mathcal A$ with 
$A_{r,R}\subset A_1$ and $\e{\beta(r-1)}<r$, $\e{\beta (R-1)} >R$ 
such that 
\[ \Re \left ( \e{i\theta} \frac{\tau'( \rho \e{i\theta})}{  \tau( \rho \e{i\theta})} \right ) \geq \beta 
  \quad (\forall \rho\in [r,R], \forall \theta \in \mathbb R)\,.\]
Equation~(\ref{e:rde}) now implies 
\[
\log \abs{\tau(\e{i\theta})}-\log\abs{\tau(r \e{i\theta})}= \Re \int_r^1 \e{i \theta} \frac{\tau'(\rho  \e{i \theta})}{\tau(\rho  \e{i\theta})}\,d\rho 
\geq \beta (1-r)
\]
and
\[
\log \abs{\tau(R\e{i\theta})}-\log\abs{\tau(\e{i\theta})}= \Re \int_1^R \e{i \theta} \frac{\tau'(\rho  \e{i \theta})}{\tau(\rho  \e{i\theta})}\,d\rho 
\geq \beta (R-1)\,.
\]
Thus 
\[
\abs{\tau (r \e{i \theta})} \leq \e{\beta(r-1)}<r  \qquad \text{and} \qquad \abs{\tau(R \e{i \theta})} \geq \e{\beta (R-1)} >R\,,
\]
so $A_0:=A_{r,R}$ has all the desired properties. 
\end{proof}

Given an expanding circle map $\tau$, we associate with it a transfer operator $\tr$ by setting 
\begin{equation}\label{ldef}
(\tr f)(z) = \sum_{k=1}^K \phi'_k(z)(f\circ\phi_k)(z) \quad (z \in \mathbb T)\,,
\end{equation}
where $\phi_k$ to denotes the $k$-th local inverse of $\tau$. 
It turns out that $\mathcal L$ is well-defined and bounded as an 
operator on $L^1(\mathbb T)$. Although this is a standard result, we shall provide a short proof, since 
part of the argument will play a crucial role later on. In the following we shall use $\mathcal C_\rho$ to denote a simple closed positively 
oriented path along the circle centred at the origin with radius $\rho$.  

\begin{lem} \label{lem:dual}
For any $f\in L^1(\mathbb T)$ and any $g\in L^\infty(\mathbb T)$ we have 
\begin{equation}\label{lem:dual:e}
\frac{1}{2\pi i}\int_{\mathcal C_1}(\tr f)(z)\cdot g(z)\,dz =  \frac{1}{2\pi i}\int_{\mathcal C_1}f(z)\cdot (g\circ \tau)(z) \,dz\,. 
\end{equation}
In particular, the transfer operator $\tr $ is bounded as an operator from $L^1(\mathbb T)$ to $L^1(\mathbb T)$. 
\end{lem}
\begin{proof} Using change of variables we see that 
\begin{multline*}
\frac{1}{2\pi i}\int_{\mathcal C_1} (\tr f)(z) \cdot  g(z)\, dz
= \sum_{k=1}^K \frac{1}{2\pi i}\int_{\mathcal  C_1} \phi'_k(z) \cdot ( f\circ \phi_k )(z) \cdot g(z)\, dz \\
= \sum_{k=1}^K  \frac{1}{2\pi i}\int_{\phi_k({\mathcal C_1})} f(z) \cdot (g\circ \tau)(z) \, dz 
= \frac{1}{2\pi i}\int_{\mathcal C_1} f(z) \cdot (g\circ \tau)(z) \, dz\,,
\end{multline*}
where we have used the fact that 
$\cup_{k=1}^K\phi_k({\mathcal C_1})={\mathcal C_1}$ up to a set of measure zero.  
Finally, the assertion that $\tr$ maps $L^1(\mathbb T)$ continuously
into itself follows from 
equation~(\ref{lem:dual:e}), since for any $f\in L^1(\mathbb T)$ and  any $g\in L^\infty(\mathbb T)$ 
\begin{equation*}
\abs{\frac{1}{2\pi i}\int_{\mathcal C_1}(\tr f)(z)\cdot g(z)\,dz} \leq 
\norm{f}{L^1(\mathbb T)}\norm{g\circ \tau}{L^\infty(\mathbb T)}=
\norm{f}{L^1(\mathbb T)}\norm{g}{L^\infty(\mathbb T)}\,. 
\qedhere
\end{equation*}
\end{proof}

It turns out that $\mathcal L$ leaves certain subspaces of
$L^1({\mathbb T})$ invariant. Of particular interest are spaces 
consisting of holomorphic functions. 
\begin{defn}
For $U$, an open subset of ${\mathbb C}$,   
we write 
\[ H^\infty(U) = \all{f:U\to{\mathbb C}}{f \text{ holomorphic and }  
\sup_{z \in U }|f(z)|<\infty} \]
to denote the Banach space of bounded holomorphic
functions on $U$ equipped with the norm 
 $\norm{f}{H^\infty(U)}=\sup_{z \in U}|f(z)|$.
\end{defn}
The proof of the invariance of $H^\infty(U)$ under $\mathcal L$ for 
suitable $U$ will rely on Fourier theory. Here and in the following, we shall
use 
\begin{equation}
c_n(f)=\frac{1}{2 \pi i}\int_{\C{C}_1} \frac{f(z)}{z^{n+1}}\, dz
\end{equation}
to denote the $n$-th Fourier coefficient of $f\in L^1({\mathbb T})$. 

Before stating the next result we require some more notation. 
Given two subsets $U$ and $V$ of $\mathbb C$ we write 
\[ U\cc V\]  if
$\cl{U}$ is a compact subset of $V$. 

We now have the following result. 

\begin{lem}\label{lem:Ltilde}
Suppose that annuli $A$ and $A'$ in ${\mathcal A}$ have been 
chosen\footnote{This is always possible by Lemma~\ref{lem:expand}.} 
such that 
\begin{equation}
\label{accap}
A_0 \cc A' \subset A \text{ and } \tau(\partial A_0)\cap \cl{A}=\emptyset\,.
\end{equation} 
Then the transfer operator $\tr $ maps $H^\infty(A')$ continuously to $H^\infty(A)$. 
\end{lem}
\begin{figure}
\includegraphics[trim= 0 -30 0 -70,width=.308\paperwidth]{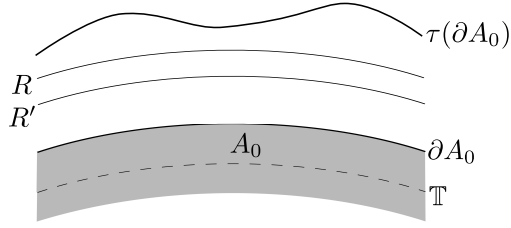}
\caption{Proof of Lemma \ref{lem:Ltilde}: choice of annuli $A_0,A'$ and $A$. }
\label{fig:Annuli}
\end{figure}

\begin{proof}
Given $f\in H^\infty(A')$, we
shall show that ${\mathcal L}f\in H^\infty(A)$ by estimating the
asymptotic behaviour 
of the Fourier coefficients of ${\mathcal L}f$. 
Write $R_0$ and $R$ to denote the radii of the circles forming the 
`exterior' boundary of $A_0$ and $A$, respectively 
(see Figure \ref{fig:Annuli}). 
Next choose
$R''$ with 
\[ \inf_{z\in \C{C}_{R_0}}|\tau(z)|>R''>R\,.\]
Similarly, write $r_0$ and $r$ to denote the radii of the circles forming the 
`interior' boundary of $A_0$ and $A$, respectively and choose 
$r''$ with 
\[ \sup_{z\in \C{C}_{r_0}}|\tau(z)|<r''<r\,.\]

Fix $f\in H^\infty(A')$ with $\norm{f}{H^\infty(A')}\leq 1$ and let $n\geq 0$. Using Lemma~\ref{lem:dual} we see that
\begin{multline*}
 |c_n(\C{L} f)| = 
\left|\frac{1}{2 \pi i} \int_{\C{C}_1} \frac{f(z)}{\tau(z)^{n+1}}\,dz\right| =
\left|\frac{1}{2 \pi i} \int_{\C{C}_{R_0}} \frac{f(z)}{\tau(z)^{n+1}}\,dz\right| \\
\leq \frac{1}{2 \pi} \int_{\C{C}_{R_0}} \frac{1}{|\tau(z)|^{n+1}}\,|dz| 
\leq \frac{R_0}{(R'')^{n+1}}\,.  
\end{multline*}
Similarly, for $n\geq1$ we have 
\begin{multline*}
 |c_{-n}(\C{L} f)| = 
\left|\frac{1}{2 \pi i} \int_{\C{C}_1} \frac{f(z)}{\tau(z)^{-n+1}}\, dz\right| =
\left|\frac{1}{2 \pi i} \int_{\C{C}_{r_0}} f(z)\tau(z)^{n-1}\, dz\right| \\
\leq \frac{1}{2 \pi} \int_{\C{C}_{r_0}}
|\tau(z)|^{n-1}\, |dz| 
\leq r_0(r'')^{n-1}\,.  
\end{multline*}
Hence, $\sum_{n=0}^{\infty} c_n(\tr f)z^n$ converges absolutely for all 
$|z|\leq R$
and $\sum_{n=1}^{\infty} c_{-n}(\tr f)z^{-n}$ converges absolutely for all
$|z|\geq r$. Moreover, for $z\in A$ we have  
\begin{multline*} 
 \abs{\sum_{n=-\infty}^{\infty} c_n(\tr f)z^n} 
 \leq  \sum_{n=0}^{\infty} \abs{c_n(\tr f)} R^n + 
            \sum_{n=1}^{\infty} \abs{c_{-n}(\tr f)}r^{-n} \\
  \leq  \sum_{n=0}^{\infty} \frac{R_0R^n}{(R'')^{n+1}} + 
            \sum_{n=1}^{\infty} \frac{r_0(r'')^{n-1}}{r^n}           
  =\frac{R_0}{R''-R}+\frac{r_0}{r-r''}\,.
 \end{multline*}
  Thus, by the uniqueness of the Fourier transform on $L^1(\mathbb T)$, we conclude that 
 $\tr f\in H^\infty(A)$ and 
 \[ \norm{\tr f}{H^\infty(A)}\leq \left ( \frac{R_0}{R''-R}+\frac{r_0}{r-r''}  \right )\norm{f}{H^\infty(A')} \quad (\forall f\in H^\infty(A'))\,. \qedhere\] 
 \end{proof}

Choosing $A=A'$ in the previous lemma shows that $\tr$ induces a well
defined continuous operator from $H^\infty(A)$ to itself. 
It turns out that $\tr: H^\infty(A)\to H^\infty(A)$ is compact. In
order to prove this result, we shall employ a factorisation argument.   

Given $A,A'\in \C{A}$ with $A'\cc A$ define the canonical
embedding $J:H^\infty(A)\rightarrow H^\infty(A')$ 
\begin{equation}\label{eq:J}
(Jf)=f|_{A'}\,.
\end{equation}
The embedding $J$ is compact by Montel's Theorem (see, for example, 
\cite[Chapter~7, Theorem 2.9]{Conway73}) and, as we shall shortly see, 
is well-approximated by the following operators: 
for $N$ a positive integer, define the finite rank operator  
$J_{N}:{H^\infty}(A)\rightarrow{H^\infty}(A')$ by
\begin{equation}\label{eq:J_N1N2}
(J_{N} f)(z)= \sum_{n=-N-1}^{N-1} c_{n}(f) z^{n} \quad \text{ for }  z\in A'\,.
\end{equation}
The approximability of $J$ alluded to earlier is the content of the
following result. 
\begin{lem}\label{lem:J}
 Let $J$ and $J_{N}$ be defined as above.
 Then 
  \[
 \lim_{N\to \infty}\norm{J-J_{N}}{H^\infty(A)\rightarrow H^\infty(A')}=0\,.
 \] 
In particular, the embedding $J$ is compact. 
\end{lem}

\begin{proof} Choose $A''\in \C{A}$ with 
\[ A'\cc A'' \cc A\,.\]
Let $R'$ and $r'$ denote the radii of the circle forming the `exterior'
and `interior' boundary of $A'$, respectively,  
and let $\C{C}_{R''}$ and $\C{C}_{r''}$ denote the oriented `exterior'
and `interior' boundary of $A''$, respectively, so that 
\[ r''<r'<R'<R''\,.\] 
Fix $f\in H^\infty(A)$ with $\norm{f}{H^\infty(A)}\leq 1$. Then 
\begin{align*}
  \norm{Jf-J_{N}f}{H^\infty(A')}
  &=\sup_{z\in A'}\left |\sum_{n\geq N} \frac{z^n}{2\pi i} 
  \int_{\mathcal C_1}\frac{f(\zeta)}{\zeta^{n+1}}\, d\zeta +  
   \sum_{n\geq N+2}\frac{z^{-n}}{2\pi i} 
  \int_{\mathcal C_1}\frac{f(\zeta)}{\zeta^{-n+1}}\, d\zeta \right |\\
  &\leq \sum_{n\geq N} \frac{(R')^n}{2\pi} \int_{\C{C}_{R''}} 
    \frac{|f(\zeta)|}{|\zeta|^{n+1}}\, |d\zeta| + 
    \sum_{n\geq N+2} \frac{(r')^{-n}}{2\pi} 
   \int_{\C{C}_{r''}} \frac{|f(\zeta)|}{|\zeta|^{-n+1}}\,|d\zeta|\\
  &\leq \left(\frac{R'}{R''}\right )^{N}\frac{1}{1-\frac{R'}{R''}}
     +
    \left(\frac{r''}{r'}\right )^{N+2}\frac{1}{1-\frac{r''}{r'}}
\end{align*}
from which the assertions follow. 
\end{proof}

Now, the transfer operator $\tr:{H^\infty}(A)\rightarrow {H^\infty}(A)$, 
factorises as 
\begin{equation}\label{eq:L}
 \tr = \tilde{\tr}J
\end{equation}
where $J:H^\infty(A)\to H^\infty(A')$ is the canonical embedding,
which is compact by Lemma~\ref{lem:J}, and 
$\tilde{\tr}$ is the transfer operator viewed as an operator 
from $H^\infty(A')$ to $H^\infty(A)$, guaranteed to be continuous by 
Lemma~\ref{lem:Ltilde}. Thus, the factorisation (\ref{eq:L}) 
implies the following result.   
\begin{prop}\label{cor:compact}
Let $A\in {\mathcal A}$ with $A_0\cc A$. Then $\tr:H^\infty(A)\to
H^\infty(A)$ is compact.
\end{prop}


\section{A family of circle maps}\label{sec:circle}
In this section we introduce a family of analytic expanding circle maps 
for which we are able to explicitly determine the spectrum of $\mathcal{L}.$
As already mentioned in the introduction,  
the family arises from the construction of an expanding map for which
the transfer operator has a specified eigenfunction for a 
given eigenvalue $\lambda$.  
A short calculation reveals that for 
$\lambda \in (-1,1)$ a solution to \eqref{eq:cos} 
lifts to $\Phi\colon\B{R}\to\B{R}$, where 
\begin{equation}\label{eq:phi}
 \Phi(x)= \frac{x}{2} - \frac{1}{\pi} \arccos{\left(\lambda 
\cos \left(\frac{\pi x}{2} \right)\right).}
\end{equation}
This is  an increasing diffeomorphism with inverse
$F\colon\mathbb{R}\to \mathbb{R}$ given by 
\begin{equation}\label{eq:F}
 F(x)= 2x+1+\frac{2}{\pi} \arctan \Big (\frac{\lambda \sin(\pi x)}{1-\lambda \cos(\pi x)} \Big)\,.
\end{equation}
Note that $F$ is a lift of a circle map $\tau:\mathbb{T}\to\mathbb{T}$,
which satisfies $F(x+2) = F(x)+4$ and $p \circ F = \tau \circ p,$ where $p:\mathbb{R}\to \mathbb{T}$ 
is the projection map defined by $p(x)= e^{i\pi x}.$
The map $\tau$ is a twofold covering of $\B{T}$
(see Figure~\ref{fig:Map}). 
Note that $F'>1$ for $\lambda \in (-1,1)$. Thus $\tau$
is an analytic expanding circle map.  

It turns out that  $\tau$ can be written in closed form. 
Using the relation $e^{i\pi F(x)} = \tau (e^{i\pi x})$, it follows
that, for $\lambda \in \mathbb{R}$,  
\[ \pi F(x) =  \pi x + \arg(\frac{\lambda - e^{i\pi x}}{1-\lambda e^{i\pi x}}) = 
2\pi x + \pi + 2\arg(1-\lambda e^{-i \pi x})\,,\]
which gives
\begin{equation}\label{eq:tau}
 \tau(z)=\frac{z (\lambda-z)}{1-\overline{\lambda} z}  \quad \text{for } z \in \mathbb{T}\,.
\end{equation}
It is not difficult to see  
that the above expression for $\tau$ yields an analytic expanding circle
map not just for real $\lambda$, but for any $\lambda \in \mathbb{C}$ 
with $|\lambda|<1$ (see Figure~\ref{fig:Map}).  
It is possible to write down lifts of (\ref{eq:tau}) for complex
$\lambda$ similar to \eqref{eq:F}. In fact, a short calculation shows
that if $\lambda = |\lambda|e^{i\alpha}$ 
then the argument of $\arctan$ in \eqref{eq:F} needs to be replaced by 
$|\lambda| \sin(\pi x-\alpha)/ (1-|\lambda| \cos(\pi x-\alpha))$.
\begin{figure}[h!]

\includegraphics[trim= 0 -20 0 -20,width=.5\paperwidth]{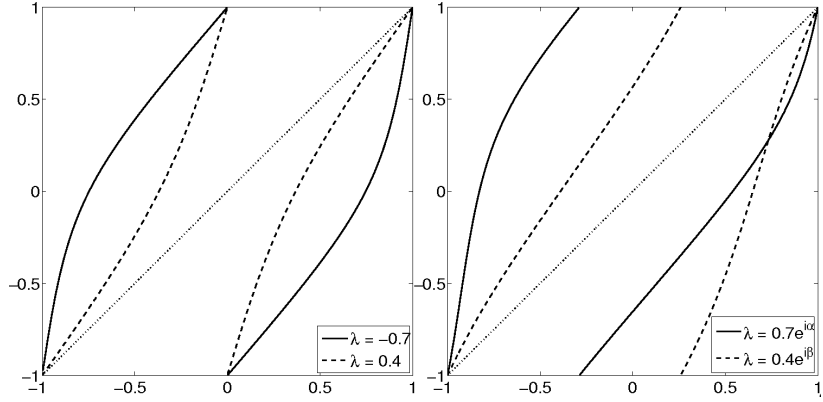}

\caption{\label{fig:Map}$\tau$ projected on the interval 
$[-1,1]$ for (left) $\lambda = -0.7$ and $\lambda = 0.4$ and
(right) $\lambda = -0.3-i\sqrt{0.4} = 0.7e^{i\alpha}$ with $\alpha \approx -2.0137$
and $\lambda = 0.1+i\sqrt{0.15} = 0.4e^{i \beta}$ with $\beta \approx 1.318.$
Note that the projection is chosen such that the interval endpoint
$-1$ is fixed by $\tau$.}
\end{figure}
Given $\tau$ as in (\ref{eq:tau}), we now choose an annulus 
$A \in \C{A}$ with $A_0\cc A$. By Proposition~\ref{cor:compact} the 
associated transfer operator 
$\mathcal{L}:{H}^\infty(A) \to {H}^\infty(A)$ is well-defined and
compact. 

As we shall see, the spectrum of $\mathcal{L}$ can be computed by
analysing the spectrum of a suitable 
matrix representation, which is obtained as follows.  
For $N\in\B{N}$ consider the projection $P_N$  
given by the same functional 
expression as $J_N$ in \eqref{eq:J_N1N2}, now viewed as an 
operator from $H^\infty(A)$ to  $H^\infty(A)$. 
Clearly, $P_N\tr P_N$ is an operator of rank $2N+1$. 
Writing $e_n(z)=z^n$, the set 
$\all{e_n}{-N-1\leq n \leq N-1}$ 
is a basis for 
\[ H_N= P_N(H^\infty(A))\] 
and the restriction of 
$P_N\C{L}P_N$ to $H_N$ is represented by the
$(2N+1) \times (2N+1)$ matrix 
$L^{(N)}$ defined by

\begin{equation}\label{eq:L^N}
\begin{split}
(L^{(N)})_{n,l} = c_{n-1}(\C{L}e_{l-1})
=\frac{1}{2\pi i} \int_{\C{C}_1} \frac{z^l}{\tau(z)^{n}}\frac{dz}{z} 
=\frac{1}{2\pi i} \int_{\C{C}_1} z^{l-n} \Big (\frac{1-\overline{\lambda}z}{\lambda-z} 
\Big)^n \frac{dz}{z}.
\end{split}
\end{equation} 
In particular, the non-zero spectrum of 
$P_N\C{L}P_N$ is given by the non-zero spectrum of $L^{(N)}$. 

Observe that \eqref{eq:L^N} defines an infinite matrix $L$ containing  
$L^{(N)}$ as a finite
submatrix. The following lemma summarizes the properties of $L$.

\begin{lem}\label{lem:L_explicit}
For $l,n \in \mathbb{Z}$ the following hold:
\begin{enumerate}[(a)]
 \item $L_{0,0}= 1$; 
 \item $L_{0,l}= 0$  if $l\neq0$; 
 \item $L_{-n,-l} = \overline{L_{n,l}} $; 
 \item $L_{-n,-n} = \lambda^n$  for $n\geq 0 $; 
 \item $L_{n,l} = L_{-n,-l} = 0$ for $n>l$. 
 \end{enumerate}
 
\end{lem}

\begin{proof}
Assertions (a) and (b)
immediately follow from \eqref{eq:L^N}, \\
while (c) is a consequence of    
\begin{multline*}
 L_{-n,-l} = 
\frac{1}{2\pi i} \int_{\C{C}_1} z^{n-l} 
  \Big (\frac{\lambda-z}{1-\overline{\lambda}z} \Big)^n \frac{dz}{z}=
\frac{1}{2\pi}\int_{0}^{2\pi} e^{i\theta(n-l)} 
  \Big (\frac{\lambda-e^{i\theta}}{1-\overline{\lambda}e^{i\theta}} \Big)^n \,d\theta \\
= 
\frac{1}{2\pi } \int_{0}^{2\pi} e^{-i\theta(l-n)} 
  \Big (\frac{1-\lambda e^{-i\theta}}{\overline{\lambda}-e^{-i\theta}} \Big)^n \,d\theta 
= \overline{L_{n,l}}\,.
\end{multline*}
For (d) and (e), observe that 
$z\mapsto \left( \frac{\lambda-z}{1-\overline{\lambda}z}\right)$ is holomorphic for all $z$ in the closed
unit disk. Thus, by the Residue Theorem,  
$$L_{-n,-n}=\frac{1}{2\pi i} \int_{\mathcal{C}_1}
  \frac{1}{z}\Big (\frac{\lambda-z}{1-\overline{\lambda}z} \Big)^n \, dz=\lambda^n.$$
Finally $n>l$ implies $L_{-n,-l}=0$, as the intergrand is a 
holomorphic function.
\end{proof}

The lemma above 
implies that $L^{(N)}$ has the following upper-lower triangular matrix
structure\footnote{Note that 
the matrix elements can be computed explicitly.
For $n>0$ and $l\in \mathbb{Z}$ we have 
\[ 
  (L^{(N)})_{n,l} 
  = \begin{cases} 
    (-\overline{\lambda})^{2n-l}  \sum_{m=0}^{l-n} {l-m-1 \choose n-1} {n \choose m} (-|\lambda|^2)^{l-n-m} & \text{if  $l-n\leq n$} \\
    \lambda^{l-2n}  \sum_{m=0}^{n} {l-m-1 \choose n-1} {n \choose m} (-|\lambda|^2)^{n-m} & 
    \text{if  $l-n\geq n.$}
  \end{cases}
\]
}
\begin{equation*}
L^{(N)}=\left(
\begin{array}{ccccccc}
 \lambda^N & 0 & 0 & 0 & 0 & \ldots & 0 \\
 \vdots & \ddots & 0 & 0 & \vdots & \ddots & \vdots \\
 * & * & \lambda & 0 & 0 & \ldots & 0 \\
 0 & 0 & 0 & 1 & 0 & 0 & 0 \\
 0 & \ldots & 0 & 0 & \overline{\lambda} & * & * \\
 \vdots & \ddots & \vdots & 0 & 0 & \ddots & \vdots \\
 0 & \ldots & 0 & 0 & 0 & 0 & \overline{\lambda}^N
\end{array}
\right).
\end{equation*}
Clearly, the spectrum 
of $L^{(N)}$ is given by the diagonal elements
$(L^{(N)})_{n,n}$, that is, 
\[\sigma(L^{(N)}) = \{\lambda^n: n=0,\ldots, N\} \cup \{\overline{\lambda}^n:n=1,\ldots, N\}.\]
Moreover, 
the triangular structure of $L^{(N)}$ implies 
$\mathcal{L}(H_N) \subseteq H_N$. Before embarking on the proof of our
main result, we require one more 
lemma, which 
relates the eigenvalues of $\tr$ with the eigenvalues of $L^{(N)}$. 

\begin{lem}\label{cor:eigenval}
Let $A\in {\mathcal A}$ with $A_0\cc A$ and 
suppose that $\tr(H_N)\subseteq H_N$ for every $N\in\B{N}$. 
Then the non-zero eigenvalues (with multiplicities) of 
$\tr:H^\infty(A)\to H^\infty(A)$ are precisely
the non-zero eigenvalues of $L^{(N)}$ as $N$ tends to infinity.
In particular, the spectrum of $\tr$ is given by 
\[ \sigma(\tr) =
\cl{\bigcup_{N\in \B{N}}{\sigma(\C L|_{H_N})}}=
\cl{\bigcup_{N\in \B{N}}{\sigma(L^{(N)})}}\,.
\]
\end{lem}
\begin{proof}
Clearly $\sigma(L^{(N)})\subseteq \sigma(\tr)$. 
Since 
$\tr(H_N)\subseteq H_N$ we have 
$\tr P_N=P_N\tr P_N$ for every $N\in\B{N}$. Using the factorisation~\eqref{eq:L}
we see that  
 \begin{align*}
\|\tr -P_N\tr P_N\|_{H^\infty(A)\to H^\infty(A)}
&=\|\tr -\tr P_N\|_{H^\infty(A)\to H^\infty(A)}\\
&=\|\tilde{\tr}J-\tilde{\tr}JP_N\|_{H^\infty(A)\to H^\infty(A)}\\
&=\|\tilde{\tr}(J-J_N)\|_{H^\infty(A)\to H^\infty(A)}\\
&\leq\|\tilde{\tr}\|_{H^\infty(A')\to H^\infty(A)}
\|(J-J_N)\|_{H^\infty (A)\to H^\infty(A')}\,,
\end{align*}
which, using Lemmas~\ref{lem:Ltilde} and \ref{lem:J}, implies 
\[
\lim_{N\to\infty}
\|\C{L}-P_N\mathcal{L}P_N\|_{{H^\infty}(A)\to{H^\infty}(A)}=0\,.
\]
This,
together with \cite[XI.9.5]{DS} guarantees that 
every non-zero eigenvalue of $\tr$ is an eigenvalue of $L^{(N)}$ for 
some $N\in \B{N}$.
\end{proof}
We are now able to prove our main result. 
\begin{proof}[Proof of Theorem \ref{thm:main}]
 For every $\lambda \in \B{C}$ with $|\lambda| < 1$, the map $\tau$ in \eqref{eq:tau} is an analytic
 expanding circle map. By Lemma \ref{cor:eigenval} the spectrum of the
 associated $\mathcal{L}:H^\infty(A)\to H^\infty(A)$ 
consists of eigenvalues, together with zero, given by
\[ \sigma(\C{L}) =
\cl{\bigcup{\sigma(\C L|_{H_N})}}=\cl{\bigcup{\sigma(L^{(N)})}}=
\{\lambda^n : n \in \mathbb{N}_0\}\cup \{\overline{\lambda}^n : n\in \mathbb{N}\} 
\cup \{0\}\,.\]
The assertions concerning the multiplicities of the eigenvalues of
$\tr$ follow from the corresponding properties of $L^{(N)}$.  
\end{proof}


\section{Circle maps considered on an interval}\label{sec:Interval}
In the previous section we have considered the transfer operator 
$\mathcal{L}_\mathbb{T}:{H^\infty}(A)\rightarrow {H^\infty}(A)$ 
associated to an analytic expanding circle map $\tau:\mathbb{T}\rightarrow\mathbb{T},$ which maps
the space of bounded holomorphic functions on an appropriately chosen annulus $A\in\C{A}$
compactly into itself. 
The circle map $\tau$ gives rise to a map $T$ on 
an interval $I = [x_0,x_1]$, chosen such that a fixed point $z_0$ of 
$\tau$ corresponds to the interval endpoint $x_0$.
Choosing a suitable complex neighbourhood $D$ of $I$, we shall now
study the spectral properties of $\mathcal{L}_I:H^\infty(D)\to
H^\infty(D)$, the transfer operator corresponding to $T$. 



More precisely, let $T:I\to I$ denote the interval map arising from
the circle map 
$\tau:\mathbb{T}\to \mathbb{T}$ via $p\circ T = \tau \circ p$ 
with a projection $p:I\to \mathbb{T}$ satisfying $p(x_0) = z_0$.\footnote{
A suitable choice is 
$p(x)=e^{2i\pi \frac{(x-x_0)}{(x_1-x_0)}+ i \arg{z_0}}$.} 
Let $\{\Phi_k \colon k=1,\ldots,K \}$ be the set of 
inverse branches of $T$. With slight abuse of notation we 
keep writing $T$ and $\Phi_k$ for the respective
analytic extensions to neighborhoods containing $I$.
Since $\tau$ is an analytic K-covering, we have 
the matching conditions 
(with suitable labelling of the
inverse branches)
\begin{equation}\label{eq:fixpt}
\begin{aligned}
\Phi_1(x_0) &= x_0\,, &
\Phi_K(x_1) &= x_1\,, &
\Phi_1^{(n)}(x_0) = \Phi_K^{(n)}(x_1)\,,\\
\Phi_{k+1}(x_0) &= \Phi_k(x_1)\,, &
\Phi^{(n)}_{k+1}(x_0)&=\Phi_k^{(n)}(x_1) &
\text{for } k=1,\ldots, K-1\,,
\end{aligned}
\end{equation}
where for each $n\in \B{N}$, we use 
$\Phi_k^{(n)}$ to denote the $n$-th derivative of
$\Phi_k$. 

Since $T$ is expanding, all inverse branches $\Phi_k$ are contractions.
We can thus 
choose a topological disk $D$ containing $I$ such that $p(D)=A$ and
$\Phi_k(D)\cc D$ for all $k$.
Then $\mathcal{L}_I: {H}^\infty(D) \to {H}^\infty(D)$, given by
\[ \mathcal{L}_If =  
\sum_{k=1}^{K} \Phi'_k\cdot (f\circ \Phi_k)\,,
\] 
yields a bounded operator. 
Moreover, $\mathcal{L}_I$ is compact (see, for example, 
\cite{BaJe_AM08,Mayer1991}),
its spectrum consisting of countably many eigenvalues
accumulating at zero only.


\begin{rem}\label{rem:Fp}
It is perhaps not surprising that 
the operators $\tr_\mathbb{T}$ and $\tr_I$ are closely related.
In order 
to see this, we define the operator $Q_p:H^\infty(A)\to H^\infty(D)$ by 
\[ (Q_pf)(x) = p'(x)f(p(x))\,.\] 
Clearly $p(D) = A$ implies
$Q_p$ injective. However, the operator 
$Q_p$ is not surjective,  
as the image $\image{Q_p} = 
\{f\in H^\infty(D): f^{(n)}(x_0) = f^{(n)}(x_1) ~\forall n\in\B{N}_0\} $
is not all of $H^\infty(D)$.
It is easy to verify that $\C{L}_I$ and $\C{L}_\B{T}$
are related by
\[ \C{L}_IQ_p = Q_p\C{L}_{\B{T}}\,,\]
and that $\sigma(\C{L}_{\B{T}})\subseteq \sigma(\C{L}_{I})$, which 
follows from the injectivity of $Q_p$.
On the other hand, an eigenvalue of
$\C{L}_{I}$ with an eigenfunction $f$ is also an eigenvalue of $\C{L}_{\B{T}}$
if $f\in \image{Q_p}$.
\end{rem}
The following lemma connects the spectrum of $\mathcal{L}_I$ 
with the spectrum of $\mathcal{L}_{\mathbb{T}}$. 
This result is mentioned in the introduction of \cite{Keller2004}
together with a proof based on the theory of Fredholm determinants. 
Here we shall give a short alternative proof. 
\begin{lem}\label{lem:interval}
Suppose that $\tau$ is 
an analytic expanding circle map and 
$T:I\to I$ the corresponding interval map fixing 
the interval endpoint $x_0$. Let 
$\mathcal{L}_{\mathbb{T}}$ and $\mathcal{L}_{I}$ be the corresponding
transfer operators as defined above.
Then the spectrum of $\mathcal{L}_I$ is given by 
\[ 
\sigma(\mathcal{L}_I) = 
\sigma(\mathcal{L_\mathbb{T}})\cup 
  \all{(T'(x_0))^{-n}}{n\in\mathbb{N}}\,.
\]
\end{lem}
\begin{proof}
Let $H^\infty(D)^*$ denote the strong dual of $H^\infty(D)$ and let  
$\mathcal{L}^*_{I}:{H}^\infty(D)^*\to {H}^\infty(D)^*$ denote the dual
operator of $\mathcal{L}_I$, that is, 
\[ (\mathcal{L}^*_{I} l )(f) = l (\mathcal{L}_{I}f) 
\quad (\forall l\in {H}^\infty(D)^*,\forall f\in {H}^\infty(D))\,.\]
For $n\in\mathbb{N}_0$, let $l_n \in {H}^\infty(D)^*$  be defined by 
\[ l_n(f)=f^{(n)}(x_1)-f^{(n)}(x_0) \quad \text{for } f\in H^\infty(D)\,.\] 
It is not difficult to see that 
$l_0$ is an eigenvector of $\mathcal{L}^*$ with
eigenvalue $\Phi^{'}_1(x_0)$ since
\begin{align*}
(\mathcal{L}^*_{I} l_0 )(f) = & (\mathcal{L}_{I}f) (x_1)
- (\mathcal{L}_{I}f)(x_0)\\
 = &  \Phi'_K(x_1) (f\circ \Phi_K)(x_1) -\Phi'_1(x_0) (f\circ \Phi_1)(x_0) \\
  & +  \sum_{k=1}^{K-1} \Phi'_k(x_1) (f\circ \Phi_k)(x_1) - \Phi'_{k+1}(x_0) (f\circ \Phi_{k+1})(x_0) \\
  = &  \Phi'_1(x_0) ( f(x_1) - f(x_0) ) \\
  = &  \Phi'_1(x_0) l_0(f)\,,
\end{align*}
where the penultimate equality follows from \eqref{eq:fixpt}.
We can proceed similarly for an arbitrary $n\in \B{N}$. Observe that 
the $n$-th derivative of $\C{L}_{I}f$ is given by 
$$(\C{L}_{I}f)^{(n)} = \sum_{k=1}^{K} \sum_{m=0}^{n-1} w_{k,m} \cdot (f^{(m)}\circ \Phi_k) +
\sum_{k=1}^{K} (\Phi'_k)^{n+1} \cdot (f^{(n)} \circ \Phi_k),$$
where each $w_{k,m}$ is
a weight function composed of derivatives of $\Phi_k$ of order up to $m+1$, and in analogy
with \eqref{eq:fixpt} satisfying $w_{k,m}(x_1) = w_{k+1,m}(x_0)$ for $k=1,\ldots, K-1$. 
A calculation similar to the above yields
\begin{align*}
(\mathcal{L}^*_{I} l_n )(f) =& (\mathcal{L}_{I}f)^{(n)} (x_1)
- (\mathcal{L}_{I}f)^{(n)}(x_0)\\
 =&  \sum_{m=0}^{n-1} w_{1,m}(x_0) l_m(f) + (\Phi'_{1}(x_0))^{n+1} l_n(f).
\end{align*}
It follows that $\mathcal{L}^*_{I} V_n \subseteq V_n$, where 
$V_n = \text{span}\{l_0,\ldots,l_n\}$ for each $n$. Thus
$(\Phi'_1(x_0))^{n}$ is an
eigenvalue of $\mathcal{L}^*_{I},$ and hence of 
$\mathcal{L}_{I}$.
As $T'(x_0) = 1/\Phi'_1(x_0)$ and every eigenvalue of $\mathcal{L}_{\mathbb{T}}$ is an eigenvalue of $\mathcal{L}_I$, we have shown 
$$\sigma(\mathcal{L_\mathbb{T}})\cup \{T'(x_0)^{-n}: n\in\mathbb{N}\}
\subseteq \sigma(\mathcal{L}_I).$$

For the converse inclusion recall Remark \ref{rem:Fp} and assume that 
$f\in {H}^\infty(D)$ is an eigenfunction of $\mathcal{L}_I$ with
eigenvalue $\mu$ and $f\notin \image{Q_p}$. 
It follows that there is $N\in\mathbb{N}_0$ such that
$f^{(N)}(x_0)\neq f^{(N)}(x_1)$ and $f^{(n)}(x_0)=f^{(n)}(x_1)$ for
$0\leq n<N$, from which 
$l_n(f) = 0$ for $0\leq n<N$.
Since $\mathcal{L}_{I}f = \mu f$, this implies
\[ 
l_N(\mu f) = l_N(\mathcal{L}_{I}f) = (\Phi'_1(x_0))^{N+1} l_N(f)\,.\]
As $l_N$ is linear and non-zero, it follows 
that $\mu = (\Phi'_1(x_0))^{N+1} = (T'(x_0))^{-N-1}$.  
\end{proof}

We can now apply this result to the interval maps arising from the 
family of circle maps defined in Section \ref{sec:circle}. 
Let $I=[-1,1]$ and $\lambda\in\B{R}$ with $|\lambda|<1$, 
then the interval map $T$
arising from $\tau$ in \eqref{eq:tau} fixes the interval endpoint
$x_0=-1$ with $1/T'(-1) = (\lambda+1)/{2}$.
By Theorem \ref{thm:main} and Lemma \ref{lem:interval}, 
the eigenvalues of $\tr_I$ can be divided into two classes, those
given by the eigenvalues of $\tr_\mathbb{T}$ (each of multiplicity two, except 
the eigenvalue $1$ of multiplicity one)   
and those by the powers of the inverse multiplier of the fixed point $x_0$, that is,
\begin{equation}\label{eq:specInt}
\sigma(\mathcal{L}_I)=\Big(\{ \lambda^n: n\in\mathbb{N}_0\}\cup
\{0\}\Big)
\cup 
\{\Big(\frac{\lambda+1}{2}\Big)^n: n\in\mathbb{N}\}\,,
\end{equation}
see also Figure \ref{fig:specInt}.

Considering $\lambda\in \B{C}$, say $\lambda=|\lambda| e^{i\alpha}$ with 
$|\lambda|<1$, we now
obtain counterexamples to Conjecture \ref{conj: LevinMayer}.
The fixed point of $\tau$ is $z_0 = (\lambda-1)/(1-\overline{\lambda})
\in \mathbb{T}$
with 
\[T'(-1) = \tau'(z_0) = \frac{\lambda + \overline{\lambda}-2}
{\lambda\overline{\lambda}-1}
= \frac{2(|\lambda|\cos(\alpha)-1)}{|\lambda|^2-1}.\]
As above, the spectrum of $\tr_I$ splits into two parts:
\begin{align*}\label{eq:specInt_complex}
\sigma(\mathcal{L}_I) =& \Big(\{\lambda^n : n\in \mathbb{N}_0\}
\cup \{\overline{\lambda}^n: n\in \mathbb{N}\} 
\cup \{0\} \Big) \cup \\
&\{ \Big(\frac{|\lambda|^2-1}{2(|\lambda|\cos(\alpha)-1)}\Big)^n:
n\in \mathbb{N}\}.
\end{align*}
Note that the transfer operator $\mathcal{L}_I$ associated to $T$
satisfies the conditions of Conjecture \ref{conj: LevinMayer}, 
but, for $\lambda\notin \mathbb{R}$, has countably infinitely many non-real eigenvalues of arbitrarily
small modulus.

\begin{figure}[!h]
\center{\includegraphics[width=.45\paperwidth]{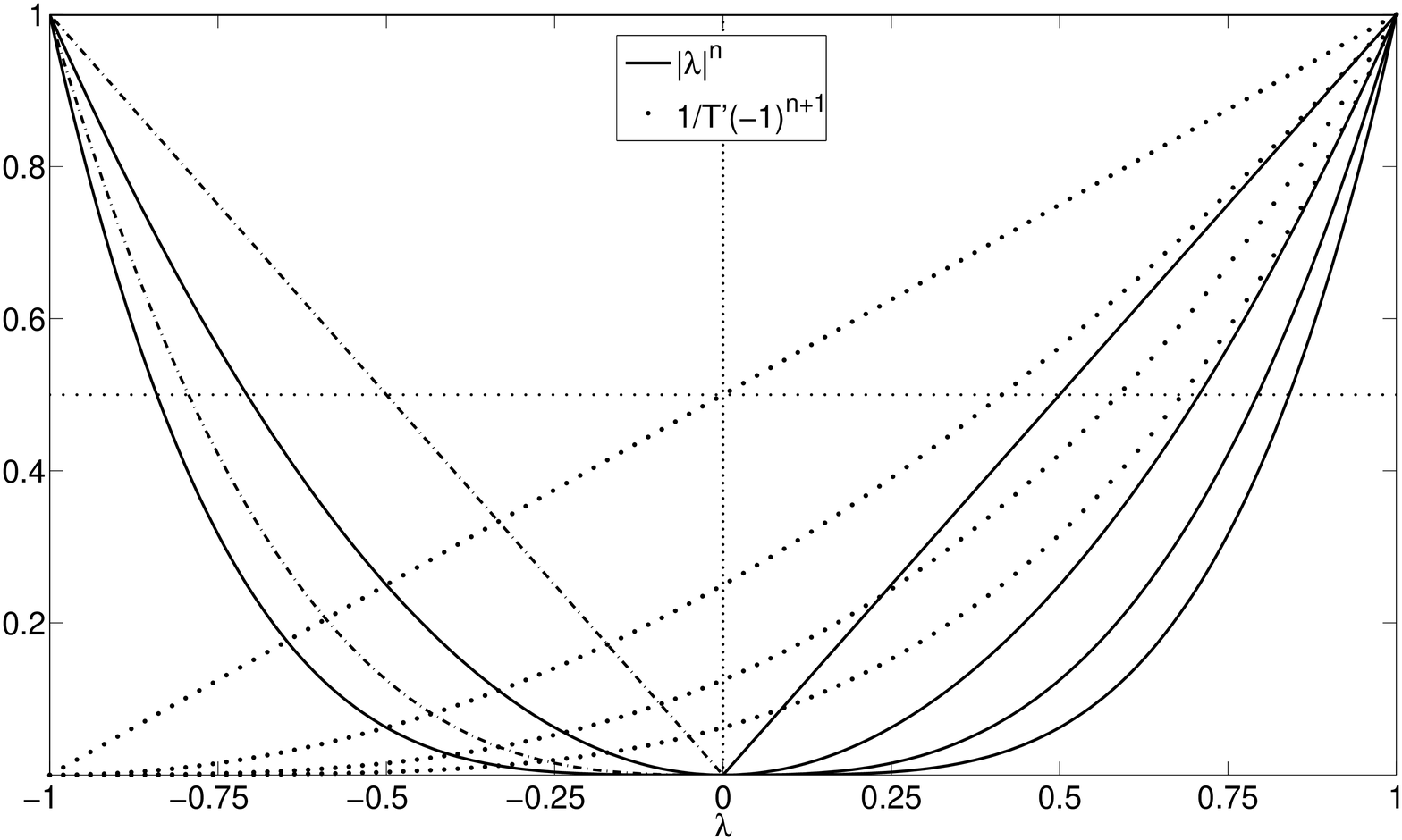}} 
 \caption{
For each $\lambda\in (-1,1)$ and for $n = 0,\ldots,4$ the eigenvalues in the spectrum \eqref{eq:specInt} of $\C{L}_I$ are plotted (in modulus). 
These are comprised of the eigenvalues $\lambda^n$ 
of $\C{L}_{\B{T}}$ (solid line for $\lambda^n > 0$, and dashed for $\lambda^n < 0$ ) and the eigenvalues
$1/(T'(-1))^{n+1}$ of $\C{L}_{I}$.
Note that the case $\lambda = 0$ corresponds to the 
doubling map. 
}
\label{fig:specInt}     
\end{figure} 

\bigskip

To the best of our knowledge these are the first examples of
nontrivial circle and interval maps for which the entire spectrum
of the associated Perron-Frobenius operators is known explicitly.
Certain conjectures were previously hard to test,  
without examples. These might now be more accessible.

\section*{Acknowledgements}
We would like to thank Hans-Henrik Rugh for pointing out to us that
transfer operators of circle maps require a holistic approach. 
W.J. gratefully acknowledges support by EPSRC (grant no.\ EP/H04812X/1).
%
%
%
%
%
%
%

\appendix
\setcounter{section}{1}
\renewcommand\thesection{\Alph{section}}


\begin{thebibliography}{10}

\bibitem{Bala:00}
V.~Baladi.
\newblock {\em {Positive transfer operators and decay of correlations}}.
\newblock World Scientific, Singapore, 2000.

\bibitem{BaJe_AM08}
O.~F. Bandtlow and O.~Jenkinson.
\newblock {Explicit eigenvalue estimates for transfer operators acting on
  spaces of holomorphic functions}.
\newblock {\em Adv. Math.}, 218(3):902--925, 2008.

\bibitem{BaJe_ETDS}
O.~F. Bandtlow and O.~Jenkinson.
\newblock {On the Ruelle eigenvalue sequence}. 
\newblock {\em Ergodic Theory Dynam. Systems}, 28(6):1701--1711, 2008. 

\bibitem{BaJe_LMS}
O.~F. Bandtlow and O.~Jenkinson.
\newblock {Invariant measures for real analytic expanding maps}. 
\newblock {\em J. Lond. Math. Soc.}, 75(2):343--368, 2007. 

\bibitem{BoyaGora:97}
A.~Boyarsky and P.~Gora.
\newblock {\em {Laws of Chaos: Invariant Measures and Dynamical Systems in One
  Dimension (Probability and its Applications)}}.
\newblock Birkh\"{a}user, Basel, 1997.

\bibitem{Conway73}
J.~B. Conway.
\newblock {\em Functions of One Complex Variable}. 
\newblock Springer, New York, 1973. 

\bibitem{Diakonos1996}
F.~Diakonos and P.~Schmelcher.
\newblock {On the construction of one-dimensional iterative maps from the
  invariant density: the dynamical route to the beta distribution}.
\newblock {\em Phys. Lett. A}, 211(4):199--203, 1996.


\bibitem{DS}
N.~Dunford and J.~T. Schwartz.
\newblock {\em {Linear Operators, Part 2: Spectral Theory}}.
\newblock Wiley-Interscience, New York, 1963.

\bibitem{Ershov1988}
S.~V. Ershov and G.~G. Malinetskii.
\newblock {The solution of the inverse problem for the Perron-Frobenius
  equation}.
\newblock {\em USSR Computational Mathematics and Mathematical Physics},
  28(5):136--141, 1988.

\bibitem{Gora1993}
P.~Gora and A.~Boyarsky.
\newblock {A matrix solution to the inverse Perron-Frobenius problem}.
\newblock {\em Proc. Amer. Math. Soc.}, 118(2):409--414, 1993.

\bibitem{Keller2004}
G.~Keller and H.~H. Rugh.
\newblock {Eigenfunctions for smooth expanding circle maps}.
\newblock {\em Nonlinearity}, 17(5):1723--1730, 2004.

\bibitem{Levin1994}
G.~M. Levin.
\newblock {On Mayer's conjecture and zeros of entire functions}.
\newblock {\em Ergodic Theory Dynam. Systems}, 14(03):565--574, 1994.

\bibitem{Mayer1987}
D.~H. Mayer. 
\newblock {On the location of poles of Ruelle's zeta function}. 
\newblock {\em Lett. Math. Phys.}, 14(2):105--115, 1987. 

\bibitem{Mayer1991}
D.~H. Mayer.
\newblock {Continued fractions and related transformations}.
\newblock In {\em Ergodic Theory, Symbolic Dynamics and Hyperbolic Spaces},
  pages 175--222. Oxford University Press, Oxford, 1991.

\bibitem{ruelle76}
D.~Ruelle.
\newblock {Zeta-functions for expanding maps and Anosov flows}.
\newblock {\em Invent. Math.}, 34(3):231--242, 1976.

\bibitem{SBJ}
J.~Slipantschuk, O.~F. Bandtlow, and W.~Just. 
\newblock {On the relation between Lyapunov exponents and 
exponential decay of correlations}. 
\newblock {\em J. Phys. A: Math. Theor.}, 46(7):075101, 2013. 
\end{thebibliography}
\end{document}